\documentclass[12pt]{amsart}

\usepackage[a4paper]{geometry}
\usepackage{enumitem}
\usepackage[utf8]{inputenc}
\usepackage[english]{babel}
\usepackage[T1]{fontenc}
\usepackage{amsfonts}
\usepackage{amsmath}
\usepackage{amssymb}
\usepackage{amstext}
\usepackage{amsthm}
\usepackage{amscd}
\usepackage{cite}
\usepackage{url}
\usepackage{mathtools}
\usepackage{setspace}
\usepackage[hidelinks]{hyperref}
\usepackage[pagewise]{lineno}

\newtheorem{theorem}{Theorem}[section]
\newtheorem{corollary}[theorem]{Corollary}
\newtheorem{lemma}[theorem]{Lemma}
\newtheorem{proposition}[theorem]{Proposition}

\theoremstyle{definition}
\newtheorem{definition}[theorem]{Definition}
\newtheorem{remark}[theorem]{Remark}
\newtheorem{example}[theorem]{Example}
\newcommand{\Q}{\mathbb{Q}}

\newcommand{\OO}{\mathcal{O}}
\newcommand{\Z}{\mathbb{Z}}
\newcommand{\F}{\mathbb{F}}
\newcommand{\C}{\mathbb{C}}
\newcommand{\p}{\mathcal{P}}
\newcommand{\N}{\mathbb{N}}

\newcommand{\abs}[1]{\left|#1\right|}

\newcommand{\gpf}{\operatorname{gpf}}

\def\ord{\operatorname{ord}}
\def\End{\operatorname{End}}

\def\ord{\operatorname{ord}}

\title{Some effectivity results for primitive divisors of
	elliptic divisibility sequences}
\author{Matteo Verzobio}
\date{}
\begin{document}
	\begin{abstract}
		Let $P$ be a non-torsion point on an elliptic curve defined over a number field $K$ and consider the sequence $\{B_n\}_{n\in \N}$ of the denominators of $x(nP)$. We prove that every term of the sequence of the $B_n$ has a primitive divisor for $n$ greater than an effectively computable constant that we will explicitly compute. This constant will depend only on the model defining the curve.
	\end{abstract}
	\maketitle
	\renewcommand{\thefootnote}{}
	\renewcommand{\thefootnote}{\arabic{footnote}}
	\setcounter{footnote}{0}
	\section{Introduction}
	Let $E$ be an elliptic curve defined by the equation
	\[
	y^2+a_1xy+a_3y=x^3+a_2x^2+a_4x+a_6,
	\]
	with coefficients in a number field $K$. Let $P\in E(K)$ be a non-torsion point and let $\OO_K$ be the ring of integers of $K$.
	Let us define the fractional ideal
	\begin{equation}\label{defbn}
		\left(x(nP)\right)\OO_K=\frac{A_n}{B_n}
	\end{equation}
	with $A_n$ and $B_n$ two relatively prime integral $\OO_K$-ideals. We want to study the sequence of integral $\OO_K$-ideals $\{B_n\}_{n\in \N}$. These are the so-called \textbf{elliptic divisibility sequences}. In particular, we want to study when a term $B_n$ has a primitive divisor, i.e., when there exists a prime ideal $\p$ such that \[\p \nmid B_1B_2\cdots B_{n-1} \mbox{    but    } \p\mid B_n.\]In \cite[Proposition 10]{silverman}, Silverman proved that, if $E$ is defined over $\Q$, then $B_n$ has a primitive divisor for $n$ large enough. This result was generalized for every number field $K$ by Cheon and Hahn in \cite{ChHa}, where the following theorem is proved.
	\begin{theorem}\label{ChHat} \cite[Main Theorem]{ChHa}
		Let $E$ be an elliptic curve defined over a number field $K$ and let $P$ be a non-torsion point in $E(K)$. Consider the sequence $\{B_n\}_{n\in \N}$ of integral $\OO_K$-ideals as defined in \eqref{defbn}. Then, for all but finitely many $n\in \N$, $B_n$ has a primitive divisor.
	\end{theorem} 
	The previous theorem is not effective. Indeed, the proof relies on Siegel’s ineffective theorem about integral points
	on elliptic curves. The aim of this paper is to make the work of \cite{ChHa} effective. Indeed, we will explicitly compute a constant $C$ so that, for $n>C$, $B_n$ has always a primitive divisor.
	\begin{theorem}\label{Thm1}
		Let $E$ be an elliptic curve defined over a number field $K$ and let $P$ be a non-torsion point in $E(K)$. Consider the sequence $\{B_n\}_{n\in \N}$ of integral $\OO_K$-ideals as defined in \eqref{defbn}. There exists a constant $C(E/K,\mathcal{M})>0$, effectively computable and depending only on the curve $E$ over the field $K$ equipped
		with a model $\mathcal{M}$ also defined over $K$, such that $B_n$ has a primitive divisor for \[n>C(E/K,\mathcal{M}).\]
	\end{theorem}
	In Section \ref{expsec}, we explicitly compute such a constant $C(E/K,\mathcal{M})$ (see Equation \eqref{finalK}).
	\begin{remark}
		The dependence on the model $\mathcal{M}$ is necessary. Indeed, given a non-torsion point $P$ on an elliptic curve $E$ and a positive constant $C$, it is easy to show that we can find a model of $E$ such that $B_n$ does not have a primitive divisor for all $n\leq C$.
	\end{remark}
	\begin{remark}
		It is conjectured that, in the case when $\mathcal{M}$ is minimal, the constant $C(E/K,\mathcal{M})$ should depend only on the field $K$. In \cite[Theorem 1]{ingramsilverman} it is proved that the number of terms without a primitive divisor of an elliptic divisibility sequence can be bounded by a constant that does not depend on $E$ and $P$, in the case when $E$ is given by a minimal model, $K=\Q$, and assuming the $abc$-conjecture.
	\end{remark}
	\begin{remark}
		We believe that the techniques used in this paper can be applied also to a generalization of elliptic divisibility sequences. Let $\OO$ be the endomorphism ring of $E$ and, given $\alpha\in \OO$, define $B_\alpha$ as the denominator of $(x(\alpha P))\OO_K$. The sequence $\{B_\alpha\}_{\alpha\in \OO}$ is a sequence of ideals and one can give a definition of primitive divisors also for these sequences (see \cite[Section 1]{EDSCM}). It has been shown in \cite[Main Theorem]{EDSCM} that also in this case there are only finitely many terms that do not have a primitive divisor (see also \cite{shiftedcm}). In the case when $\End(E)=\Z$ this is a trivial corollary of Theorem \ref{ChHat}, but in the case  $\End(E)\neq \Z$ (i.e. when $E$ has complex multiplication) this is far from being easy. We believe that using the techniques of this paper one can find an explicit upper bound for the degree of $\alpha$ such that $B_\alpha$ does not have a primitive divisor, in the case when $\End(E)$ is a maximal order and it is a Principal Ideal Domain.
	\end{remark}
	\section*{Acknowledgements}
	This paper is part of the author's PhD thesis at Università of Pisa. Moreover, this project has received funding from the European Union’s Horizon 2020 research and 
	innovation programme under the Marie Skłodowska-Curie Grant Agreement No. 
	101034413. 
	
	I thank the referee for many helpful comments.
	\section{Notation}\label{sec:not}
	The following notation will be used during the paper. The curve $E$ is defined by the equation
	\[
	y^2+a_1xy+a_3y=x^3+a_2x^2+a_4x+a_6
	\]
	with coefficients in the number field $K$.
	\begin{itemize}[leftmargin=*]
		\item[] $\Delta$ is the discriminant of the equation defining the curve;
		\item[] $\Delta_{E/K}$ is the minimal discriminant of the elliptic curve;
		\item[] $j(E)$ is the $j$-invariant of the curve;
		\item[] $D=[K:\Q]$ is the degree of the number field $K$;
		\item[] $\OO_K$ is the ring of integers of $K$;
		\item[] $\Delta_K$ is the discriminant of the field $K$;
		\item[] $\mathit{f}_{E/K}$ is the conductor of the curve;
		\item[] $
		\sigma_{E/K}=\frac{\log\abs{\N_{K/\Q}(\Delta_{E/K})}}{\log\abs{\N_{K/\Q}(\mathit{f}_{E/K})}},
		$ where $\N_{K/\Q}$ is the norm of the field extension, is the Szpiro quotient; if $E/K$ has everywhere good reduction (and then $\mathit{f}_{E/K}=1$), we put $\sigma_{E/K}=1$;
		\item[] if $x\in \OO_K$ is non-zero, then $\gpf(x)$ is the greatest rational prime $p$ so that $\ord_p(\N_{K/\Q}(x))>0$;
		\item[] if $n\in \N$ is non-zero, then $\omega(n)$ is the number of rational prime divisors of $n$;
		\item[] if $x\in K^*$, define $\mathfrak{m}(x)=\max_{\p}\{\ord_\p(x)\}$ where the maximum runs over all primes in $\OO_K$.
	\end{itemize}
	\section{Preliminaries}
	Let $M_K$ be the set of all places of $K$, take $\nu \in M_K$, and let $\abs{\cdot}_\nu$ be the absolute value associated with $\nu$. Let $n_\nu$ be the degree of the local extension $K_\nu/\Q_\nu$. We normalize the absolute values as in \cite[Section VIII.5, after Example VIII.5.1]{arithmetic}. If $\nu$ is finite, then $\abs{p}_\nu=p^{-1}$, where $p$ is the rational prime associated to $\nu$. If $\nu$ is infinite, then $\abs{x}_\nu=\max\{x,-x\}$ for every $x\in \Q$. Thanks to this choice, we have the usual product formula, i.e.
	\[
	\prod_{\nu \in M_K}\abs{x}_\nu^{n_\nu}=1
	\]
	for every $x\in K^*$. Define $M_K^\infty$ as the set of infinite places of $K$ and $M_K^0$ as the set of finite places.
	
	Now, we define the height of a point on the curve; more details can be found in \cite[Chapter VIII]{arithmetic}. Given $x\in K^*$, define
	\[
	h_\nu(x):=\text{max}\left\{0,\log \abs{x}_\nu\right\}
	\]
	and
	\[
	h(x):=\frac{1}{[K:\Q]}\sum_{\nu \in M_K} n_\nu h_\nu(x).
	\]
	For every point $R\neq O$ of $E(K)$, define
	\[
	h_\nu(R):=h_\nu(x(R))
	\]
	and the height of the point as
	\[
	h(R):=h(x(R)).
	\]
	So, for every $R\in E(K)\setminus \{O\}$,
	\[
	h(R)=\frac{1}{[K:\Q]}\sum_{\nu \in M_K} n_\nu h_\nu(R).
	\]
	Finally, put
	\[
	h(O)=0,
	\]
	where $O$ is the identity of the curve.
	
	Given a point $R$ in $E(K)$, define the canonical height as in \cite[Proposition VIII.9.1]{arithmetic}, i.e.
	\[
	\hat{h}(R)=\frac 12\lim_{N\to \infty}\frac{h\left(2^NR\right)}{4^N}.
	\] 
	
	We recall the properties of the height and of the canonical height that will be necessary for this paper.
	\begin{itemize}[leftmargin=*]
		\item It is known that the difference between the height and the canonical height can be bounded by an explicit constant. In particular, we will use the following result. Let
		\[
		C_E=\frac{h(j(E))}{4}+\frac{h(\Delta)}{6}+2.14.
		\]
		If $E$ is defined by a Weierstrass equation in short form and with integer coefficients, then, for every $R\in E(K)$,
		\[
		\abs{h(R)-2\hat{h}(R)}\leq C_E.
		\]
		This is proved in \cite[Equation 3]{Silvermancanonical}.
		\item The canonical height is quadratic, i.e.
		\[
		\hat{h}(nR)=n^2\hat{h}(R)
		\]
		for every $R$ in $E(K)$ and $n\in \N$.
		\item For every non-torsion point $R\in E(K)$,
		\[
		\hat{h}(R)>0.
		\]
		There exists a positive constant $J_E$, effectively computable and depending only on $E$ and $K$, such that
		\[
		J_E\leq \hat{h}(P)
		\]
		for every non-torsion point $P\in E(K)$. Thanks to \cite[Theorem 2]{petsche}, we can take 
		\[
		J_E=\frac{\log\abs{\N_{K/\Q}(\Delta_{E/K})}}{10^{15}D^3\sigma_{E/K}^6\log^2\left(104613D\sigma_{E/K}^2\right)}
		\]
		where $\N_{K/\Q}$ is the norm of the field extension, $D=[K:\Q]$, and
		\[
		\sigma_{E/K}=\frac{\log\abs{\N_{K/\Q}(\Delta_{E/K})}}{\log\abs{\N_{K/\Q}(\mathit{f}_{E/K})}}.
		\]
		If $\mathit{f}_{E/K}=1$, we put $\sigma_{E/K}=1$. 
		The conductor $\mathit{f}_{E/K}$ is defined in \cite[beginning of Section VIII.11]{arithmetic}.
	\end{itemize}
	
	In order to prove that $B_n$ has a primitive divisor for all but finitely many terms, Silverman in \cite{silverman} and Cheon and Hahn in \cite{ChHa}, used a Theorem of Siegel that says
	\[
	\lim_{n\to \infty}\frac{h_\nu(nP)}{h(nP)}=0
	\]
	for every $\nu\in M_K$, as is proved in \cite[Theorem IX.3.1]{arithmetic}. This result is not effective and hence their results are not effective. We will use some results that tell us effectively how this limit goes to $0$. As we will show later, for the finite places we will use some results on the formal group of the elliptic curve, and for the infinite places we will use the work of David in \cite{David}. The idea of using the result of David to study primitive divisors of elliptic divisibility sequences has been introduced, as far as we know, by Streng in \cite[Section 3]{EDSCM}.
	
	We conclude this section by showing that we can focus only on the case when $E$ is defined by a Weierstrass equation in short form and with integer coefficients. We will do that in Lemma \ref{lemma:cov}. In order to prove that lemma, we need the following.
	
	\begin{lemma}\label{lemma:boundorder}
		Let $E$ be an elliptic curve defined over $K$ by a Weierstrass equation with integer coefficients and let $P\in E(K)$. Let $\nu\in M_K^0$, $\p$ be the associated prime, and $p$ be the associated rational prime. There exists
		\[
		k\leq p^{\frac{\nu(\Delta(E))}{12}}(2\N_{K/\Q}(\p)+1)\max\{4,\ord_\p(j(E)^{-1})\}
		\]
		such that $\nu(x(kP))<0$.
	\end{lemma}
	\begin{proof}
		Let $E_\p$ be a minimal model for the elliptic curve over $K_\p$ and let $P_\p$ be the image of $P$ under the change of variables from $E$ to $E_\mathcal{P}$. So, $x(P)=u_\p^2x(P_\p)+r_\p$ for some $u_\p,r_\p\in K_\p$. By \cite[Proposition VII.1.3.d]{arithmetic}, $\nu(u_\p)\geq 0$ and $\nu(r_\p)\geq 0$. Note that $12\nu(u_\p)=\nu(\Delta(E))-\nu(\Delta(E_\p))\leq \nu(\Delta(E))$ and so $\nu(u_\p)\leq \nu(\Delta(E))/12$. There exists a multiple $kP_\p$ with $k\leq \max\{4,\ord_\p(j(E_\p)^{-1})\}$ such that $kP_\p$ is not a singular point in $E_\p(\F_\p)$ (see \cite[Corollary C.15.2.1]{arithmetic}). Observe that $\F_\p$ has $\N_{K/\Q}(\p)$ elements and then the group of non-singular points in $E_\p(\F_\p)$ has at most $2\N_{K/\Q}(\p)+1$ elements. So, the order of $kP_\p$ in the group of non-singular points modulo $\p$ is at most $2\N_{K/\Q}(\p)+1$. Hence, there exists \[n_\p(P_\p)\leq (2\N_{K/\Q}(\p)+1)\max\{4,\ord_\p(j(E_\p)^{-1})\}\] such that $n_\p(P_\p)P_\p$ reduces to the identity modulo $\p$. Given a point $Q$ in $E(K)$, it is easy to show that $Q$ reduces to the identity modulo $\p$ if and only if $\nu(x(Q))<0$. Therefore, $\nu(x(n_\p(P_\p)P_\p))<0$. 
		
		From a classic result on formal groups, \[\nu\left(x(p^{\nu(u_\p)}n_\p(P_\p)P_\p)\right)<-2\nu(u_\p).\] For more details on formal groups, see Lemma \ref{lemma:boundnu} or \cite[Corollary IV.4.4]{arithmetic}. Using that $\nu(u_\p)\geq 0$ and $\nu(r_\p)\geq 0$, we have
		\begin{align*}
			\nu\left(x(p^{\nu(u_\p)}n_\p(P_\p)P)\right)&=\nu\left(u_\p^2x(p^{\nu(u_\p)}n_\p(P_\p)P_\p)+r_\p\right)\\&=\nu\left(x(p^{\nu(u_\p)}n_\p(P_\p)P_\p)\right)+2\nu(u_\p)\\&<0.
		\end{align*}
		We conclude recalling that $\nu(u_\p)\leq \nu(\Delta(E))/12$.
	\end{proof}
	
	\begin{lemma}\label{lemma:cov}
		Let $E/K$ be an elliptic curve defined over $K$ by a Weierstrass model $\mathcal{M}$. 
		Then, there exists an elliptic curve $E'$ defined over $K$ by a short Weierstrass model $\mathcal{M}'$ with integer coefficients that is isomorphic over $K$ to $E$, and a positive rational integer $s(E/K,\mathcal{M})$ such that:
		if Theorem 1.2 holds with $C(E'/K,\mathcal{M}')$ for $E',\mathcal{M}'$, then it holds with \[C(E/K,\mathcal{M}) =\max\{C(E'/K,\mathcal{M}'), s(E/K,\mathcal{M})\}\] for $E,\mathcal{M}$. The constant $s(E/K,\mathcal{M})$ is effectively computable and will be defined during the proof (see Equation \eqref{se}). It depends only on $E$ and $\mathcal{M}$.
	\end{lemma}
	\begin{proof}
		Recall that $E$ is defined by the equation
		\[
		y^2+a_1xy+a_3y=x^3+a_2x^2+a_4x+a_6.
		\]
		Let $u$ be the smallest positive rational integer such that, after the change of variables,
		\[
		(x,y)\to (x',y')=\Big(u^2\Big(x+\frac{a_1^2}{12}+\frac{a_2}{3}\Big),u^3\Big(y+\frac{a_1}{2}x+\frac{a_3}{2}\Big)\Big)
		\]
		we have that $E$ is isomorphic to a curve $E'$ of the form
		$y'^2=x'^3+ax'+b$ with $a$ and $b$ in $\OO_K$. Let $P'$ be the image of $P$ under this isomorphism. So, $x'=u^2x+r$ for $u\in \Z_{\neq 0}$ and $r\in K$. Let $q$ be the integral $\OO_K$-ideal such that, for every $\nu\in M_K^0$,
		\[
		\nu(q)= \max\left\{\abs{\nu(u^2)},-\nu(r)\right\}.
		\]
		If $r=0$, we take $q$ such that, for every $\nu\in M_K^0$, $\nu(q)=\abs{\nu(u^2)}$.
		Note that $q$ depends only on $E$ and $\mathcal{M}$.
		Let $B_n'$ be the elliptic divisibility sequence associated with $E'$ and $P'$.
		
		Let $\nu$ be the absolute value associated with a prime $\p$ coprime with $q$. We have $\nu(u^2)=0$ and $\nu(r)\geq 0$.
		If $\p$ divides $B_n$, then $\nu(x(nP))<0$ and
		\[
		\nu\left(u^2x(nP)+r\right)=\nu\left(u^2x(nP)\right)=\nu(x(nP))=-\nu(B_n)<0.
		\]
		Therefore,
		\[
		\nu(B_n')=\nu(B_n)>0.
		\]
		In the same way, if $\p$ divides $B_n'$, then
		\[
		\nu(B_n')=\nu(B_n)>0.
		\]
		So, if $\p$ is coprime with $q$, then $\p$ divides $B_n$ if and only if it divides $B_n'$.
		
		Let 
		\begin{equation}\label{se}
			s=\max_{\p\mid q}\Big\{p^{(12\nu_\p(u)+\nu_\p(\Delta(E)))/12}(2\N_{K/\Q}(\p)+1)(\max\{4,\ord_\p(j(E')^{-1})\})\Big\},
		\end{equation}
		where $p$ is the rational prime associated with $\p$.
		
		Assume $n>s$. We will show that, if $\p$ is a primitive divisor of $B_n'$, then $\p$ is a primitive divisor also for $B_n$.
		
		Let $\p$ be a primitive divisor of $B_n'$. Suppose that $\p$ divides $q$. By Lemma \ref{lemma:boundorder}, there exists
		\begin{align*}
			k&\leq p^{(\nu(\Delta(E')))/12}(2\N_{K/\Q}(\p)+1)(\max\{4,\ord_\p(j(E')^{-1})\})\\&=p^{(12\nu(u)+\nu(\Delta(E)))/12}(2\N_{K/\Q}(\p)+1)(\max\{4,\ord_\p(j(E')^{-1})\})
			\\&\leq s
		\end{align*}
		such that $\nu(B_{k}')>0$.
		But, since $\p$ is a primitive divisor of $B_n'$ we know that $k\geq n$. Hence, $n\leq s$ and this is absurd since we assumed $n>s$. So, $\p$ does not divide $q$. Since $\p$ is a primitive divisor of $B_n'$ and $\p$ is coprime with $q$, then $\p$ divides $B_n$ and does not divide $B_k$ for $k<n$. Therefore, it is a primitive divisor for $B_n$.
		
		In conclusion, if $n>\max\{C(E'/K,\mathcal{M}'),s\}$, then $B_n'$ has a primitive divisor $\p$. As we showed, $\p$ is also a primitive divisor for $B_n$. Therefore, $B_n$ has a primitive divisor for all $n>\max\{C(E'/K,\mathcal{M}'),s\}$. 
		
		Observe that $s$ depends on $j(E')$, $\Delta(E)$, $u$, $r$, and $q$. It is easy to show these five values depend only on $E$ and the model defining the curve. So, we are done.
	\end{proof}
	
	From now on, we will assume that $E$ is defined by a short Weierstrass equation with coefficients in $\OO_K$ of the form
	\[y^2=x^3-(g_2/4)x-(g_3/4).\]
	Once we prove Theorem \ref{Thm1} under this assumption, then we can prove it in general using Lemma \ref{lemma:cov}.
	It is useful to have $E$ in this form in order to apply the work of David in \cite{David}, as we will do in Section \ref{secinf}.
	\section{Structure of the proof}
	We start by recalling the structure of the proof of Cheon and Hahn of Theorem \ref{ChHat}.
	\begin{enumerate}[leftmargin=*]
		\item\label{1} If $\p$ is a non-primitive divisor of $B_n$, then $\p$ divides $B_{n/q}$ for $q$ a prime divisor of $n$. Moreover, if $\nu$ is the place associated to $\p$, then $h_\nu(nP)$ and $ h_\nu\Big(\frac nq P\Big)$ are roughly the same.
		\item\label{sumfin} If $B_n$ does not have a primitive divisor, then, for every $\nu\in M_K^0$, we have
		\[
		h_\nu(nP)\leq\sum_{q\mid n} h_\nu\Big(\frac nq P\Big)+O(\log n),
		\]
		using Step \eqref{1}.
		Therefore,
		\[
		\sum_{\nu\in M_K^0}h_\nu(nP)\leq\sum_{\nu\in M_K^0}\left(\sum_{q\mid n} h_\nu\Big(\frac nq P\Big)+O(\log n)\right).
		\]
		\item\label{suminf} For every $\nu$ infinite, $h_\nu(nP)$ is negligible compared to $h(nP)$. In particular,
		\[
		\sum_{\nu\in M_K^\infty} h_\nu(nP)=o\left(n^2\right).
		\]
		\item\label{4} Putting together the inequalities of \eqref{sumfin} and \eqref{suminf}, we obtain
		\begin{align*}
			2n^2\hat{h}(P)&=2\hat{h}(nP)\\&=h(nP)+O(1)\\&=\frac 1D\sum_{\nu\in M_K^0} n_\nu h_\nu(nP)+\frac 1D\sum_{\nu\in M_K^\infty}n_\nu h_\nu(nP)+O(1)\\&\leq\sum_{q\mid n}h\Big(\frac nq P\Big)+o\left(n^2\right)\\&=2\hat{h}(P)\sum_{q\mid n}\frac {n^2}{q^2}+o\left(n^2\right)\\&=
			2n^2\hat{h}(P)\left(\left(\sum_{q\mid n}\frac 1{q^2}\right)+o(1)\right).
		\end{align*}  
		Note that one can use even sharper arguments using a complete inclusion-exclusion to find better inequalities (see for example \cite[Proof of the main theorem]{EDSCM}).
		\item For every $n$ we have $\sum_{q\mid n}q^{-2}<1$ and then the inequality of \eqref{4} does not hold for $n$ large enough. So, $B_n$ does not have a primitive divisor only for finitely many $n\in \N$.
	\end{enumerate}
	In order to make this proof effective, we need to make Steps \eqref{1} and \eqref{suminf} effective. In Section \ref{secfin}, we bound $h_\nu(nP)-h_\nu((n/q)P)$ as in Step \eqref{1}. In Section \ref{secinf}, we make effective Step \eqref{suminf}.
	\section{Finite places}\label{secfin}
	Take $\p$ a prime over a valuation $\nu\in M_K^0$. Let $p$ be the rational prime under $\p$. Recall that $E$ is defined by a Weierstrass equation with integer coefficients. The group of points of $E(K_\p)$ that reduce to the identity modulo $\p$ is a group that is isomorphic to a formal group, as proved in \cite[Proposition VII.2.2]{arithmetic}. Observe that, in the hypotheses of this proposition, there is the requirement that $E$ is in minimal form. Anyway, the proof works in the exact same way only requiring that the coefficients of $E$ are integers in $K_\p$, that is our case. Let $Q\in E(K_\mathcal{P})$ and, using the equation defining the elliptic curve, it is easy to show that $3\nu(x(Q))=2\nu(y(Q))$ and therefore
	\begin{equation}\label{xq}
		2\nu\left(\frac{x(Q)}{y(Q)}\right)=-\nu(x(Q))>0.
	\end{equation}
	Define
	\[
	z(Q)=\frac{x(Q)}{y(Q)}\in K_\mathcal{P}.
	\]
	\begin{lemma}\label{np}
		Take $\nu\in M_K^0$ and let $\p$ be the associated prime. Define $n_\p$ as the smallest integer such that $n_\p P$ reduces to the identity modulo $\p$. Then, $kP$ reduces to the identity modulo $\p$ if and only if $k$ is a multiple of $n_\p$. Moreover, $\nu(x(kP))<0$ if and only if $k$ is a multiple of $n_\p$.
	\end{lemma}
	\begin{proof}
		Let $E_{\text{ns}}(\F_\p)$ be the group of non-singular points of the curve $E$ reduced (with respect to the given model) modulo $\p$. 
		Suppose by contradiction that $kP$ reduces to the identity but $k$ is not a multiple of $n_\p$. Take $q$ and $r$ the quotient and the remainder of the division of $k$ by $n_\p$. Since $n_\p$ does not divide $k$, we have that $0<r<n_\p$. So,
		\[
		rP\equiv nP-kqP\equiv O-O\equiv O \mod{\p}
		\]
		and this is absurd since $n_\p$ is the smallest positive integer such that $n_\p P\equiv O\mod{\p}$. Vice versa, if $k=qn_\p$, then
		\[
		kP\equiv q(n_\p P)\equiv qO\equiv O \mod{\p}.
		\]
		Now, we conclude by observing that a point $Q$ reduces to the identity modulo $\p$ if and only if $\nu(x(Q))<0$.
	\end{proof}
	\begin{lemma}\label{lemma:boundnu}
		Let $Q\in E(K)$ be such that $\nu(z(Q))>0$. Recall that $p$ is the rational prime such that $\nu(p)>0$. Then $\nu\left(z\left(p^eQ\right)\right)\geq e+\nu(z(Q))$. In particular, if $p^e\mid n$, then $\nu(z(nQ))>e$.
	\end{lemma}
	\begin{proof}
		By \cite[Corollary IV.4.4]{arithmetic}, $\nu\left(z\left(pQ\right)\right)\geq 1+\nu(z(Q))$. Now, we proceed by induction. The case $e=0$ is trivial. Assume that we know that $\nu\left(z\left(p^{e-1}Q\right)\right)\geq e-1+\nu(z(Q))$. Put $Q'=p^{e-1}Q$ and for the observation at the beginning of the proof we know $\nu\left(z\left(pQ'\right)\right)\geq 1+\nu\left(z\left(Q'\right)\right)$. Therefore,
		\[
		\nu\left(z\left(p^eQ\right)\right)=\nu\left(z\left(pQ'\right)\right)\geq 1+\nu\left(z\left(Q'\right)\right)=1+\nu\left(z\left(p^{e-1}Q\right)\right)\geq e+\nu(z(Q)).
		\]
		Now, we deal with the second part of the lemma. Let $n=p^{e}n'$ and, by Lemma \ref{np}, $\nu(z(n'Q))>0$. For the first part of the lemma, $\nu(z(nQ))\geq e+\nu(z(n'Q))>e$.
	\end{proof}
	\begin{lemma}\label{lemma:formal}
		Let $Q\in E(K)$ be such that $\nu(z(Q))>\nu(p)/(p-1)$. Then, \[\nu(z(nQ))=\nu(z(Q))+\nu(n)\] for all $n\geq 1$.
	\end{lemma}
	\begin{proof}
		This follows by \cite[Theorem IV.6.4 and Proposition VII.2.2]{silverman}.
	\end{proof}
	\begin{definition}
		Let $S$ be the set of finite places of $K$ such that $\nu| 2$ or $\nu$ ramifies over $\Q$. Observe that this set is finite.
	\end{definition}
	\begin{corollary}\label{cor:formal}
		Let $Q\in E(K)$ be such that $\nu(z(Q))>0$. If $\nu\notin S$, then \[\nu(z(nQ))=\nu(z(Q))+\nu(n)\] for all $n\geq 1$.
	\end{corollary}
	\begin{proof}
		Since $\nu\notin S$, we have $\nu(p)=1$ and $p-1\geq 2$. So, $\nu(z(Q))\geq 1>\nu(p)/(p-1)$ and we apply Lemma \ref{lemma:formal}.
	\end{proof}
	\begin{proposition}\label{prop:nu}
		Let $E$ be an elliptic curve defined over a number field $K$ and let $P\in E(K)$ be a non-torsion point. Take $\nu\in M_K^0$, let $\p$ be the associated prime, and $p$ be the rational prime under $\p$. Recall that $n_\p$ is the smallest positive integer such that $n_\p P$ reduces to the identity modulo $\p$. Assume that $n_\p\mid n$ and $n_\p\neq n$. Then, one of the following hold:
		\begin{itemize}[leftmargin=*]
			\item There exists a prime $q\mid n$ such that $\nu\left(z\left( (n/q) P\right)\right)>0$ and
			\[
			\nu(z(nP))=\nu\left(z\left(\frac nq P\right)\right)+\nu(q);
			\]
			\item $\nu\in S$ and \[n<n_\p p^{\frac{\nu(p)}{p-1}+1}.\]
		\end{itemize}
	\end{proposition}
	\begin{proof}
		Assume $\nu\notin S$ and let $Q=n_\p P$. Since $n/n_\p$ is an integer greater than $1$, there is a prime $q$ that divides it. By Corollary \ref{cor:formal},
		\[
		\nu(z(nP))-\nu\left(z\left(\frac nq P\right)\right)=\nu\left(z\left(\frac {n}{n_\p}Q\right)\right)-\nu\left(z\left(\frac {n}{qn_\p}Q\right)\right)=\nu\left(\frac {n}{n_\p}\right)-\nu\left(\frac {n}{qn_\p}\right)=\nu(q).
		\]
		
		So, we focus on the case $\nu \in S$.
		Assume that there exists $q\neq p$ such that $q\mid n/n_\p$. Then,
		\[
		\nu(z(nP))=\nu\left(z\left(\frac nq P\right)\right)
		\]
		by \cite[Corollary IV.4.4]{silverman} and we are done since $\nu(q)=0$. Assume now that there is no $q\neq p$ such that $q\mid n/n_\p$. So, $n=p^en_\p$ with $e\geq 1$ (since $n\neq n_\p$) and recall that we defined $Q=n_\p P$. 
		
		Assume that $e-1\geq\nu(p)/(p-1)$. Then, by Lemma \ref{lemma:boundnu}, $\nu\left(z\left(p^{e-1}Q\right)\right)>e-1\geq \nu(p)/(p-1)$. Therefore, by Lemma \ref{lemma:formal},
		\[
		\nu(z(nP))=\nu\left(z\left(p^eQ\right)\right)=\nu\left(z\left(p^{e-1}Q\right)\right)+\nu(p)=\nu\left(z\left(\frac np P\right)\right)+\nu(p).
		\]
		
		It remains the case $e-1<\nu(p)/(p-1)$. In this case, 
		\[
		n=n_\p p^e<n_\p p^{\frac{\nu(p)}{p-1}+1}. 
		\]
	\end{proof}
	\begin{remark}
		To explicitly compute $\nu(z(nP))$ in the second case of the previous proposition one can use \cite[Lemma 5.1]{stange}.
	\end{remark}
	\begin{lemma}\label{lemma:C1}
		Let $\nu\in S$, $\p$ be the associated prime, and $p$ be the associated rational prime. It holds
		\[
		n_\p p^{\frac{\nu(p)}{p-1}+1}\leq \gpf(2\Delta_K)^{\frac{\mathfrak{m}(\Delta_E)}{12}}\max\{4,\mathfrak{m}(j(E)^{-1})\}\left(2\gpf(2\Delta_K)^D+1\right)\gpf(2\Delta_K)^{D+1}.
		\]
		See Section \ref{sec:not} for the definition of the constants involved.
	\end{lemma}
	\begin{proof}
		Recall that we are working with an elliptic curve $E$ defined by a Weierstrass equation with integer coefficients. By Lemma \ref{lemma:boundorder},
		\[
		n_\p\leq p^{\frac{\nu(\Delta(E))}{12}}(2\N_{K/\Q}(\p)+1)\max\{4,\ord_\p(j(E)^{-1})\}.
		\]
		Since $\p$ is a prime over a place in $S$ and the primes that ramify divide the discriminant of the field $\Delta_K$, we have $\N_{K/\Q}(\p)\leq \gpf(2\Delta_K)^D$. Therefore,\[
		n_\p\leq \gpf(2\Delta_K)^{\frac{\nu(\Delta(E))}{12}}\max\{4,\mathfrak{m}(j(E)^{-1})\}\left(2\gpf(2\Delta_K)^D+1\right).\]
		Moreover, $p\leq \gpf(2\Delta_K)$ and $\nu(p)/(p-1)\leq \nu(p)\leq D$.
	\end{proof}
	\begin{definition}\label{def:C1}
		Define
		\[
		C_1=\gpf(2\Delta_K)^{\frac{\mathfrak{m}(\Delta_E)}{12}}\max\{4,\mathfrak{m}(j(E)^{-1})\}\left(2\gpf(2\Delta_K)^D+1\right)\gpf(2\Delta_K)^{D+1}.
		\]
	\end{definition}
	\begin{proposition}\label{hnu}
		Let $E$ be an elliptic curve defined over a number field $K$ and let $P\in E(K)$ be a non-torsion point. Take $\nu\in M_K^0$ and let $\p$ be the associated prime. Assume that $n_\p\mid n$, that $n_\p\neq n$, and that $n\geq C_1$. Then, there exists a prime $q\mid n$ such that
		\[
		h_\nu(nP)=h_\nu\left(\frac nq P\right)+2h_\nu\left(q^{-1}\right).
		\]
	\end{proposition}
	\begin{proof}
		Observe that we are in the hypotheses of Proposition \ref{prop:nu}. By Lemma \ref{lemma:C1} we know that, since $n\geq C_1$, we cannot be in the second case of Proposition \ref{prop:nu}. Therefore, there exists a prime $q\mid n$ such that
		\[
		\nu(z(nP))=\nu\left(z\left(\frac nq P\right)\right)+\nu(q).
		\]
		Observe that, given $Q\in E(K)$ with $\nu(x(Q))<0$, then by Equation \eqref{xq},
		\[
		h_\nu(x(Q))=\log \abs{x(Q)}_\nu=-2\log\abs{\frac{x(Q)}{y(Q)}}_\nu=-2\log\abs{z(Q)}_\nu.
		\]
		Therefore,
		\[
		h_\nu(nP)=-2\log\abs{z(nP)}_\nu=-2\log\abs{qz\left(\frac nqP\right)}_\nu=h_\nu\left(\frac nq P\right)+2h_\nu\left(q^{-1}\right).
		\]
	\end{proof}
	
	\section{Infinite places}\label{secinf}
	We know that $2n^2\hat{h}(P)$ is close to $h(nP)$ and that
	\[
	h(nP)=\frac 1D\sum_{\nu\in M_K^0}h_\nu(nP)+\frac 1D\sum_{\nu\in M_K^\infty}h_\nu(nP).
	\]
	Thanks to the previous section, we know how to bound $h_\nu(nP)$ for $\nu$ finite in the case when $B_n$ does not have a primitive divisor. Now, we need to bound $h_\nu(nP)$ for $\nu$ infinite. We show that, for $n$ large enough, $h_\nu(nP)$ is negligible compared to $n^2\hat{h}(P)$. 
	
	Recall that we are working with an elliptic curve $E$ defined by the equation $y^2=x^3-(g_2/4)x-(g_3/4)$ with $g_2$, $g_3\in \OO_K$.
	Fix an embedding $K \hookrightarrow \C$ and consider the group of complex points $E(\C)$. We briefly recall the properties of $E(\C)$. For the details see \cite[Chapter VI]{arithmetic}. There is a unique lattice $\Lambda\subseteq \C$ such that $\C/\Lambda$ is isomorphic to $E(\C)$ via the map $\phi:z\to (\wp(z),\wp'(z)/2,1)$ (see \cite[Theorem VI.5.1]{arithmetic}).
	Thanks to \cite[Proposition 1.1.5]{advanced}, we can take $\omega_1$ and $\omega_2$ two generators of $\Lambda$ such that $\tau=\omega_2/\omega_1\in \C$ is in the fundamental domain. In particular, $\Im \tau\geq\sqrt{3}/2$, where $\Im \tau$ is the imaginary part of $\tau$. We need to make this choice in order to use \cite[Theorem 2.1]{David}.
	
	Before proceeding, we need to define some constants. Let $h=\max\{1,h(1:g_2:g_3),h(j(E))\}$, where $h(1:g_2:g_3)$ is the usual height on $\mathbb{P}^2$ (for a definition see \cite[Section VIII.5]{arithmetic}). Let \[\log V_1=\max\left\{h,\left(3\pi \right)/\left(  D\cdot\Im \tau\right)\right\},\]  \[\log V_2=\max\left\{h,\left(3\pi \abs{\omega_2}^2\right)/\left( \abs{\omega_1}^2\cdot D\cdot\Im \tau\right)\right\}.\] Let $c_1:=3.6\cdot 10^{41}$, that is the constant $c_1$ of \cite[Theorem 2.1]{David} evaluated in $k=2$.
	Define
	\begin{equation}\label{C_3}
		C_3=\max\left\{30,eh,\log V_1/D,\log V_2/D,D\right\}
	\end{equation}
	and
	\begin{equation}\label{C2}
		C_2=54\cdot c_1\cdot D^6\log V_1\log V_2.
	\end{equation}
	\begin{proposition}\label{propdavid}
		Let $E$ be an elliptic curve defined by the equation $y^2=x^3-(g_2/4)x-(g_3/4)$ for $g_2,g_3\in K$ and take $P\in E(K)$. 
		Let $z\in \C$ be so that $\phi(z)=P$ and suppose $\log n> C_3$. If $0\leq m_1,n_1,m_2,n_2\leq n$ with $n_1,n_2\neq 0$, then
		\[
		\log\abs{z-\frac{m_1}{n_1}\omega_1-\frac{m_2}{n_2}\omega_2}>-C_2n^{1/2}.
		\]
	\end{proposition}
	\begin{proof}
		In \cite[Theorem 2.1]{David}, David proved that,
		for all integers $0\leq m_1,n_1,m_2,n_2\leq n$ with $n_1,n_2\neq 0$, we have
		\[
		\log\abs{z-\frac{m_1}{n_1}\omega_1-\frac{m_2}{n_2}\omega_2}>-c_1D^6(\log BD)(\log\log B+1+\log D+h)^3 \log V_1\log V_2
		\]
		where 
		\[
		\log B:=\max\{eh,\log n,\log V_1/D,\log V_2/D\}.
		\] Since $\log n>C_3$, we have $\log n> D$, $\log n>eh>h+1$, and $\log n=\log B$. Hence,
		\[
		c_1D^6(\log BD)(\log\log B+1+\log D+h)^3 \log V_1\log V_2<C_2\log^4 n.
		\]
		Moreover, since $\log n>30$, we have
		\[
		\log^4 n<n^{1/2}
		\]
		and then
		\[
		\log\abs{z-\frac{m_1}{n_1}\omega_1-\frac{m_2}{n_2}\omega_2}>-C_2n^{1/2}.
		\]
	\end{proof}
	\section{Proof of Theorem \ref{Thm1}}
	Define 
	\[
	\rho(n)=\sum_{p|n}\frac{1}{p^2}
	\]
	and $\omega(n)$ as the number of prime divisors of $n$. It is easy to prove, by direct computation, that
	\[
	\rho(n)<\sum_{p \text{  prime  }}\frac1{p^2}< \frac 12.
	\]
	Recall that $C_1$ is defined in Definition \ref{def:C1}.
	\begin{lemma}\label{divprim}
		Let $n\geq C_1$. If $B_n$ does not have a primitive divisor, then there exists an embedding $K\hookrightarrow \C$ such that
		\[
		\max\{\log\abs{x(nP)},0\}\geq 2\hat{h}(P)n^2(1-\rho(n))-2\log n-C_E(\omega(n)+1)
		\]
		where with $\abs{x(nP)}$ we mean the absolute value in the embedding and $C_E$ is defined in Section \ref{sec:not}.
	\end{lemma}
	\begin{proof}
		Suppose that $B_n$ does not have a primitive divisor and take $\nu$ finite. Let $\p$ be the associated prime and assume $\nu(B_n)>0$. Hence, $n_\p\mid n$ but $n\neq n_\p$ since $B_n$ does not have a primitive divisor. So, using Proposition \ref{hnu}, there is a prime $q_\nu\mid n$ such that 
		\[
		h_\nu(nP)=h_\nu\left(\frac {n}{q_\nu} P\right)+2h_\nu\left(q_\nu^{-1}\right).
		\]
		Let $M_K^{0,n}$ be the set of finite places $\nu$ such that $h_\nu(nP)>0$. Therefore,
		\begin{align*}
			\sum_{\nu\in M_K^0}n_\nu h_\nu(nP)&=\sum_{\nu\in M_K^{0,n}}n_\nu h_\nu(nP)\\&\leq \sum_{\nu\in M_K^{0,n}}n_\nu h_\nu\Big(\frac n{q_\nu} P\Big)+2n_\nu h_\nu\left(q_\nu^{-1}\right)\\&\leq\Big(\sum_{q|n}Dh\Big(\frac nq P\Big)+2Dh\left(q^{-1}\right)\Big).
		\end{align*}
		Here we are using that $h_\nu(kP)\geq 0$ for all $\nu\in M_K$ and all $k\geq 1$. Thus,
		\begin{align*}
			\frac{1}{D}\sum_{\nu\in M_K^{\infty}}n_\nu h_\nu(nP)&=h(nP)-\frac{1}{D}\sum_{\nu\in M_K^0}n_\nu h_\nu(nP)
			\\&\geq 2\hat{h}(nP)-C_E-\sum_{q|n}\Big(h\Big(\frac nqP\Big)+2\log q\Big)
			\\&\geq 2\hat{h}(nP)-C_E-2\log n-\sum_{q|n}\Big(2\hat{h}\Big(\frac nqP\Big)+C_E\Big)
			\\&= 2\hat{h}(P)n^2\left(1-\sum_{q\mid n}\frac{1}{q^2}\right)-2\log n-C_E(\omega(n)+1)\\&= 2\hat{h}(P)n^2(1-\rho(n))-2\log n-C_E(\omega(n)+1).
		\end{align*}
		Since $h_\nu(nP)\geq 0$ for all $\nu\in M_K$ and  $\sum_{\nu\in M_K^{\infty}} n_\nu=D$, at least one of the $h_\nu(nP)$, for $\nu\in M_K^{\infty}$, is larger than the RHS. Recalling that \[h_\nu(x(P))=\max\{\log \abs{x(nP)}_\nu,0\}\] we conclude that \[\max\{\log\abs{x(nP)},0\}\geq 2\hat{h}(P)n^2(1-\rho(n))-2\log n-C_E(\omega(n)+1).\]
	\end{proof}
	We briefly recall the hypotheses that we are assuming. As we said in the previous section, we are assuming that $E(\C)\cong \C/\Lambda$ with the lattice $\Lambda$ generated by the complex numbers $\omega_1$ and $\omega_2$. Moreover, we are working with an elliptic curve defined by a Weierstrass equation with integer coefficients and in short form. Recall that $C_2$ is defined in \eqref{C2} and define
	\[
	C_4=2\max_{\nu\in M_K^\infty}\{\max\{\abs{x(T)}_\nu\mid T\in E(\overline{K})[2]\setminus\{O\}\}\}.
	\]
	\begin{proposition}\label{prop:divprim}
		Assume that \begin{equation}\label{assump}
			2\hat{h}(P)n^2(1-\rho(n))-2\log n-C_E(\omega(n)+1)>0,
		\end{equation} that $n\geq C_1$, and that $\log n\geq C_3$, as defined in \eqref{C_3}. If $B_n$ does not have a primitive divisor, then
		\begin{equation}\label{finin}
			\hat{h}(P)n^2\leq n^{1/2}(2C_2+4+2C_E+\log C_4).
		\end{equation}
	\end{proposition}
	\begin{proof}
		Fix the embedding $K\hookrightarrow\C$ of Lemma \ref{divprim}.
		Since $B_n$ does not have a primitive divisor, we have
		\begin{equation}\label{divprim2}
			\log\abs{x(nP)}\geq 2\hat{h}(P)n^2(1-\rho(n))-2\log n-C_E(\omega(n)+1)
		\end{equation}
		thanks to Lemma \ref{divprim} and the assumption in \eqref{assump}.
		Consider the isomorphism $\C/\Lambda\cong E(\C)$ as in Section \ref{secinf} and take $z\in \C$ in the fundamental
		parallelogram of the period lattice of $E$ such that $\phi(z)=P$. Assume \[\abs{x(nP)}\geq C_4\] and let $\delta$ be the $n$-torsion point of $\C/\Lambda$ closest to $z$ (if it is not unique, we choose one of them). Then,
		\begin{equation}\label{xnp}
			\log \abs{x(nP)}\leq -2\log \abs{nz-n\delta}+\log 8
		\end{equation}
		thanks to \cite[Lemma 8]{ingram} (here we are using the assumption $\abs{x(nP)}\geq C_4$). This Lemma is stated for $K=\Q$, but the proof works in the exact same way for $K$ number field. Since $\delta$ is an $n$-torsion point, we have 
		\[
		\delta=\frac{m_1}{n}\omega_1+\frac{m_2}{n}\omega_2
		\]
		for $0\leq m_1,m_2\leq n$.
		Using Proposition \ref{propdavid} and the assumption that $\log n>C_3$, we have
		\[
		\log \abs{z-\delta}= \log \abs{z-\frac{m_1}{n}\omega_1-\frac{m_2}{n}\omega_2}\geq -C_2 n^{1/2}.
		\]
		Applying Inequalities \eqref{divprim2} and \eqref{xnp} we have
		\begin{align}\label{dislog8}
			\log 8+ 2C_2n^{1/2}&\geq -2\log\abs{z-\delta}+\log 8\nonumber\\&= 
			2\log \abs{n}-2\log\abs{nz-n\delta}+\log 8\nonumber\\&\geq
			-2\log\abs{nz-n\delta}+\log 8\nonumber\\&\geq 
			\log\abs{x(nP)}\nonumber\\&\geq 
			2\hat{h}(P)n^2(1-\rho(n))-2\log n-C_E(\omega(n)+1).
		\end{align}
		Observe that $\omega(n)\leq \log_2 n$ and $(1-\rho(n))>0.5$. Therefore, rearranging \eqref{dislog8}, we have
		\begin{align*}
			\hat{h}(P)n^2&\leq 2\hat{h}(P)n^2(1-\rho(n))\\&\leq 2\log n+C_E(\omega(n)+1) +\log 8+ 2C_2n^{1/2}\\&\leq n^{1/2}(2C_2+4+2C_E).
		\end{align*}
		Here we are using that $n^{1/2}>\log n$ thanks to the hypothesis $\log n>C_3$.
		Recall that we obtained this inequality assuming $\abs{x(nP)}\geq C_4$.
		If $\abs{x(nP)}<C_4$, applying again \eqref{divprim2}, we have
		\begin{align*}
			\log C_4\geq& \log \abs{x(nP)}\\\geq&2\hat{h}(P)n^2(1-\rho(n))-2\log n-C_E(\omega(n)+1).
		\end{align*}
		Therefore, one can easily show that, both in the case $\abs{x(nP)}<C_4$ and in the case $\abs{x(nP)}\geq C_4$, it holds
		\[
		\hat{h}(P)n^2\leq n^{1/2}(2C_2+4+2C_E+\log C_4).
		\]
	\end{proof}
	We are now ready to prove our main theorem. We will show that Equation \eqref{finin} does not hold if $n$ is large enough. 
	\begin{proof}[Proof of Theorem \ref{Thm1}]
		Define
		\[
		C_5=J_E^{-1}(2C_2+4+2C_E+\log C_4)
		\]
		and take
		\begin{equation}\label{assn}
			n>\max\left\{C_1,C_5^{2/3},V_1,V_2,\exp(D),(\exp(eh)),e^{30}\right\}.
		\end{equation}
		We want to show that $B_n$ has a primitive divisor.
		
		Observe that, thanks to the assumption in \eqref{assn} and the definition of $C_3$ in \eqref{C_3}, we have $\log n>C_3$. Moreover,
		\begin{align*}
			n^{3/2}&>C_5
			\\&=J_E^{-1}(2C_2+4+2C_E+\log C_4)
			\\&>\hat{h}(P)^{-1}(4+2C_E)
		\end{align*}
		and then \[
		n^2>\log n \cdot \hat{h}(P)^{-1}(4+2C_E).
		\]
		Therefore, Equation \eqref{assump} holds. Finally, $n\geq C_1$. Hence, we are in the hypotheses of Proposition \ref{prop:divprim}.
		
		We assume that $B_n$ does not have a primitive divisor and we find a contradiction. Since $B_n$ does not have a primitive divisor, we know that we can apply Proposition \ref{prop:divprim} and \eqref{finin} must hold. But
		\begin{align*}
			n^{3/2}&\geq J_E^{-1}(2C_2+4+2C_E+\log C_4)\nonumber\\&\geq\frac{2C_2+4+2C_E+\log C_4}{\hat{h}(P)} .
		\end{align*}
		and then \eqref{finin} does not hold. Therefore, we find a contradiction and then $B_n$ must have a primitive divisor.
		
		In conclusion, define
		\begin{equation}\label{eq:C6}
			C_6(E/K,\mathcal{M})=\max\left\{C_1,V_1,V_2,\exp(D),\exp(eh),e^{30},C_5^{2/3}\right\}
		\end{equation}
		and $B_n$ has a primitive divisor for $n>C_6(E/K,\mathcal{M})$.
		Observe that every constant involved in the definition of $C_6(E/K,\mathcal{M})$ does not depend on $P$ and it is effectively computable (we will give more details in the next section). So, we are done.
		
		Recall that we are working under the assumption that $E$ is defined by a short Weierstrass equation with integer coefficients. In order to conclude for the general case, one has to use Lemma \ref{lemma:cov}.
	\end{proof}
	\section{Explicit computation}\label{expsec}
	Now, we explicitly write a constant $C(E/K,\mathcal{M})$ such that Theorem \ref{Thm1} holds. We assume that $E$ is defined by a short Weierstrass equation with integer coefficients, the general case can be done using Lemma \ref{lemma:cov}. Recall that we defined many constants in Section \ref{sec:not}.
	
	First of all, we show how to bound $\abs{\tau}$, as defined at the beginning of Section \ref{secinf}. Recall that we are working under the assumption that $\tau$ is in the fundamental domain. Hence, we know $\abs{\Re{\tau}}\leq 1/2$ and then we study $\Im{\tau}$, the imaginary part of $\tau$. Put $q=e^{2\pi i\tau}$ and then \[\abs{q}=e^{-2\pi\Im{\tau}}.\] So, 
	\[
	\log\abs{q}=-2\pi\Im{\tau}.
	\]
	Thanks to \cite[Lemma 5.2.b]{Silvermancanonical}, we have
	\[
	\abs{\log \abs{q}}\leq 5.7+\max\{\log \abs{j(E)},0\}.
	\]
	Therefore,
	\[
	\abs{\Im{\tau}}=\frac{\abs{\log \abs{q}}}{2\pi}\leq \frac{5.7+\max\{\log \abs{j(E)},0\}}{2\pi}.
	\]
	We obtain
	\[
	\abs{\tau}^2=\abs{\Re{\tau}}^2+\abs{\Im{\tau}}^2\leq \frac 14+\left( \frac{5.7+\max\{\log \abs{j(E)},0\}}{2\pi}\right)^2.
	\]
	Let \[\log V_1'=\max\left\{h,\left(2\sqrt{3}\pi \right)/  D\right\},\]  
	\[\log V_2'=\max\left\{h,\left(2\sqrt3\pi\left(\frac 14+\left( \frac{5.7+\max\{\log \abs{j(E)},0\}}{2\pi}\right)^2\right) \right)/D\right\},\]
	and
	\[
	C_2'=54\cdot c_1\cdot D^6\log V_1'\log V_2'.
	\]	
	By the definitions of $V_1$, $V_2$, and $C_2$ given at the beginning of Section \ref{secinf}, we have $V_1'\geq V_1$, $V_2'\geq V_2$, and $C_2'\geq C_2$.
	Hence, by Equation \eqref{eq:C6}, Theorem \ref{Thm1} holds for
	\begin{equation}\label{finalK}
		C(E/K,\mathcal{M})=\max\left\{C_1,V_1',V_2',\exp(D),\exp(eh),e^{30},\left(\frac{2C_2'+4+2C_E+\log C_4}{J_E}\right)^{2/3}\right\}
	\end{equation}
	where $h=\max\{1,h(1:g_2:g_3),h(j(E))\}$, $c_1=3.6\cdot 10^{41}$, 
	\begin{itemize}
		\item[]
		\[
		C_1=\gpf(2\Delta_K)^{\frac{\mathfrak{m}(\Delta_E)}{12}}\max\{4,\mathfrak{m}(j(E)^{-1})\}\left(2\gpf(2\Delta_K)^D+1\right)\gpf(2\Delta_K)^{D+1},
		\]
		\item[]
		\[\log V_1'=\max\left\{h,\left(2\sqrt{3}\pi \right)/ D\right\},\] 
		\item[]  
		\[\log V_2'=\max\left\{h,\left(2\sqrt3\pi\left(\frac 14+\left( \frac{5.7+\max\{\log \abs{j(E)},0\}}{2\pi}\right)^2\right) \right)/D\right\},\]
		\item[]
		\[
		C_2'=54\cdot c_1\cdot D^6\log V_1'\log V_2',
		\]
		\item[]
		\[
		C_E=\frac{h(j(E))}{4}+\frac{h(\Delta)}{6}+2.14,
		\]
		\item[]
		\[C_4=2\max\{\abs{x(T)}\mid T\in E(\overline{\Q})[2]\setminus\{O\}\},\] 
		\item[]
		\[
		J_E=\frac{\log\abs{\N_{K/\Q}(\Delta_{E/K})}}{10^{15}D^3\sigma_{E/K}^6\log^2\left(104613D\sigma_{E/K}^2\right)}.
		\]
	\end{itemize}
	\section{Examples}
	We apply our main theorem to a couple of examples.
	\begin{example}\label{ex:1}
		Let $E$ be the rational elliptic curve defined by the equation $y^2=x^3-4x+4$. In this case, $D=\Delta_K=1$, $h\approx10.23$, $j(E)=-27648/11$, $\Delta_{E/K}=-2816$, $\sigma_{E/K}\approx1.78$, and $C_4\approx 4.76$. Using Equation \eqref{finalK}, we have
		\[
		C(E/K,\mathcal{M})\approx 5.88\cdot 10^{42}< 6\cdot 10^{42}.
		\]
	\end{example}
	With our methods, even if we optimize all the estimates in the proof, we cannot hope to find a constant for Theorem \ref{Thm1} much smaller than the one of Example \ref{ex:1}. Indeed, in the definition of $c_1$ and of $J_E$ appear constants that are very large (namely $10^{41}$ and $10^{15})$ and so, even if the other constants involved are small, we cannot find a constant much smaller than $10^{38}$. In order to find better constants, one would need to have better constants in the bound of canonical height and in logarithmic approximation.
	
	Now, we present another example where we show the techniques that one can use to find the terms without a primitive divisor.
	\begin{example}
		We focus on the elliptic curve $y^2=x^3-2x$ and $P=(2,2)\in E(\Q)$. The first terms of the sequence are $B_1=1$, $B_2=2^2$, $B_3=1$, $B_4=(2^4)(3^2)(7^2)$, and $B_5=(17)^2(19)^2$. Hence, $B_1$ and $B_3$ do not have a primitive divisor. 
		For the terms that have very large indexes, we can use Theorem \ref{Thm1}. So, we apply Theorem \ref{Thm1} with $C(E/K,\mathcal{M})$ as defined in Equation \eqref{finalK}. In the definition of $C(E/K,\mathcal{M})$ we substitute $J_E$ with $0.3$. Indeed, for every rational non-torsion point of $E$, we have $\hat{h}(P)>0.3$ and $J_E$ is a constant such that $J_E<\hat{h}(P)$. The minimum of the canonical height of the rational non-torsion points of $E$ is computed in \cite{lmfdb}, where the canonical height is defined as the double of our canonical height. By Theorem \ref{Thm1} we have that, for $n\geq 2\cdot 10^{31}$, $B_n$ has a primitive divisor. 
		
		To deal with the terms with indexes smaller than $2\cdot 10^{31}$, we can use the following techniques. By \cite[Theorem 1.3]{yabutavoutier} and \cite{verzobio2020primitive}, $B_n$ has a primitive divisor for $n$ even. So, we focus on the terms with odd indexes. As an easy corollary of \cite[Lemma 3.4]{yabutavoutier}, we have that if $B_n$ does not have a primitive divisor, then $\log B_n\leq 0.18n^2$. So, we can compute the values of $B_n$ and check if the inequality holds (this is much faster than computing the factorization of the terms). As far as we know, the faster way to compute $B_n$ is to use \cite[Theorem 1.9]{recurrence}, where is proved that, for $k\geq 9$,
		\begin{equation}
			b_{k}=\frac{b_{k-2}b_{k-6}b_4^2-b_{k-4}^2b_{6}b_{2}}{b_{k-8}b_2^2}
		\end{equation}
		where $b_k=\pm\sqrt{B_k}$ for an appropriate choice of the sign (for more details, see \cite[Definition B]{recurrence}). One can check that $\log B_n>0.18 n^2$ for $4\leq n\leq 10^5$ using PARI-GP \cite{PARI2} and then $B_n$ has a primitive divisor for $4\leq n\leq 10^5$. 
		So, our bound is too large to be computationally useful and then new methods are needed to bridge the gap.
	\end{example}
	
	\normalsize
	\baselineskip=17pt
	\bibliographystyle{plain}
	\bibliography{biblio}
	MATTEO VERZOBIO, INSTITUTE OF SCIENCE TECHNOLOGY AUSTRIA, AM CAMPUS 1, 3400, KLOSTERNEUBURG, AUSTRIA\\
	\textit{E-mail address}: matteo.verzobio@gmail.com
\end{document}